\documentclass[12pt]{article}
\usepackage{amsmath,amsthm,amssymb,latexsym,color}
\theoremstyle{plain}
\newtheorem{theor}{Theorem}
\theoremstyle{remark}
\newtheorem{rem}{Remark}
\theoremstyle{plain}
\newtheorem{prop}[theor]{Proposition}
\newtheorem{cor}[theor]{Corollary}
\newtheorem{lemma}[theor]{Lemma}

\def\R{{\mathbb R}}
\def\P{{\mathbb P}}
\def\N{{\mathbb N}}

\def\Net{{\mathcal N}}

\def\dist{{\rm d}}
\def\col{{\rm col}}
\def\Exp{{\mathbb E}}

\def\supp{{\rm supp}}

\def\Proj{{\rm Proj}}
\def\spn{{\rm span}}

\def\concf{{\mathcal Q}}
\def\regm{{D_1}}
\def\irregm{{D_{2}}}
\def\mofones{{\bf 1_{N\times n}}}
\def\Gras{{\rm Gr}}

\newcommand{\trunk}[2]{{#1}_{#2}}
\newcommand{\thesubspace}[4]{V_{#1,#2}(#3,#4)}
\newcommand{\thesubspaceshort}[2]{V(#1,#2)}

\begin{document}

\title{The smallest singular value of random rectangular matrices
with no moment assumptions on entries}
\author{Konstantin E. Tikhomirov, University of Alberta, Canada}

\maketitle

\begin{abstract}
Let $\delta>1$ and $\beta>0$ be some real numbers.
We prove that there are positive $u,v,N_0$ depending only on $\beta$ and $\delta$
with the following property: for any $N,n$ such that $N\ge \max(N_0,\delta n)$, any
$N\times n$ random matrix $A=(a_{ij})$ with i.i.d.\ entries satisfying
$\sup\limits_{\lambda\in {\mathbb R}}{\mathbb P}\bigl\{|a_{11}-\lambda|\le 1\bigr\}\le 1-\beta$ and
any non-random $N\times n$ matrix $B$,
the smallest singular value $s_n$ of $A+B$ satisfies
${\mathbb P}\bigl\{s_n(A+B)\le u\sqrt{N}\bigr\}\le \exp(-vN)$.
The result holds without any moment assumptions on distribution of the entries of $A$.
\end{abstract}

\section{Introduction}

In last years, spectral properties of random matrices with fixed dimensions
(the corresponding theory is often called {\it non-asymptotic})
have attracted considerable attention
of researchers, whose efforts have been mostly concentrated on studying distributions of the largest and the smallest singular values.
For detailed information on the development of the subject, we refer the reader to surveys \cite{RV_CONGRESS}, \cite{V}.

Let $N\ge n$.
Given an $N\times n$ random matrix $A$, we employ a usual notation
$s_1(A):=\max\limits_{y\in S^{n-1}}\|Ay\|$; $s_n(A):=\inf\limits_{y\in S^{n-1}}\|Ay\|$.
A limiting result of Z.D.~Bai and Y.Q.~Yin \cite{BY} suggests that for an $N\times n$ matrix with i.i.d.\ mean zero entries
with unit variance and a finite fourth moment, its largest and smallest singular values
should ``concentrate'' near $\sqrt{N}+\sqrt{n}$ and $\sqrt{N}-\sqrt{n}$, respectively.
In the non-asymptotic setting one is interested, in particular,
in finding the weakest possible conditions on random matrices that would imply
$s_1\lesssim \sqrt{N}+\sqrt{n}$ and $s_n\gtrsim \sqrt{N}-\sqrt{n}$ with a large probability.

For a random $N\times n$ matrix $A$ with i.i.d.\ mean zero subgaussian entries, 
an elementary application of the standard $\varepsilon$-net argument
yields $s_1(A)\lesssim \sqrt{N}+\sqrt{n}$ with an overwhelming probability.
Distribution of the smallest singular value when $N\approx n$ requires a more delicate analysis.
A.~Litvak, A.~Pajor, M.~Rudelson and N.~Tomczak-Jaegermann showed in \cite{LPRT}
that if $N$ and $n$ satisfy $N/n\ge 1+c_1(\ln N)^{-1}$ then $\P\{s_n(A)\le c_2\sqrt{N}\}\le\exp(-c_3N)$,
where $c_1,c_3$ depend only on variance and subgaussian moment and $c_2$ --- on the moments and the aspect ratio $N/n$.
The approach initiated in \cite{LPRT} was further developed by M.~Rudelson and R.~Vershynin who
combined it with certain Littlewood--Offord-type theorems.
In \cite{RV_SQUARE}, Rudelson and Vershynin treated square matrices and later in \cite{RV_RECT} ---
rectangular matrices with arbitrary aspect ratio and i.i.d.\ mean zero subgaussian entries, thereby
sharpening and generalizing the result of \cite{LPRT}.
We note that the Littlewood--Offord theory has gained an important role in the study of random
matrices primarily due to T.~Tao and V.~Vu
(see, in particular, \cite{TV_ANNMATH}).

Various estimates for the extremal singular values were obtained when studying the problem of
approximating covariance matrix of a random vector by the empirical covariance matrix.
Answering a question of R.~Kannan, L.~Lov\'asz and M.~Simonovits, the authors of \cite{ALPT} 
treated log-concave random vectors.
Later, the log-concavity was replaced by weaker assumptions
(see, in particular, \cite{ALPT2}, \cite{SV}, \cite{MP}, \cite{GLPT}).

Recently, it has become apparent that different conditions are required to bound the largest and
the smallest singular value, and these two questions should be handled separately.
One of results proved by N.~Srivastava and R.~Vershynin in \cite{SV} provides a lower estimate
for the second moment of $s_n(A)$, where $A$ is an $N\times n$
matrix with independent isotropic rows satisfying a $(2+\varepsilon)$-moment condition
and certain assumptions on the aspect ratio $N/n$.
It is important to note that the conditions imposed on $A$ are too weak to imply the
``usual'' upper bound $s_1(A)\lesssim\sqrt{N}$ with a large probability \cite{LS}.
This result of \cite{SV} was strengthened by V.~Koltchinskii and S.~Mendelson in \cite{KM} under similar assumptions on the matrix.
Another theorem of \cite{KM} states the following:
given an $n$-dimensional isotropic random vector $X$ satifying
$\inf\limits_{y\in S^{n-1}}\P\{|\langle X,y\rangle|\ge \alpha\}\ge \beta$ for some $\alpha,\beta>0$, there are $C_1,c_2,c_3>0$ depending only
on $\alpha,\beta$ such that for $N\ge C_1 n$ and the $N\times n$ random matrix
$A$ with i.i.d.\ rows distributed like $X$, one has $\P\{s_n(A)\ge c_2\sqrt{N}\}\ge 1-\exp(-c_3 N)$.
Some further improvements of the estimates have been obtained in \cite{Yaskov}.

The assumption of isotropicity of a random vector or, more generally, boundedness of variance of its coordinates
is quite natural and appears as part of requirements on a matrix' rows in all the aforementioned papers.
However, for a deeper understanding of non-asymptotic characteristics of random matrices,
an important question is {\it whether any moment assumptions on entries are really necessary}
in order to get satisfactory lower estimates for the smallest singular value.

Unlike in \cite{SV} and \cite{KM} where the matrix entries within a given row are not necessarily independent,
in our paper we consider the classical setting when a rectangular matrix has i.i.d.\ entries.
However, in contrast with all the mentioned results, the lower estimate for the smallest singular value that we prove does not use any moment assumptions;
the only requirement is that the distribution of entries satisfies a ``spreading'' condition given in terms of
the Levy concentration function.
Moreover, compared to \cite{SV} and \cite{KM}, we significantly relax the assumptions on the aspect ratio of the matrix.

Given a real random variable $\xi$, {\it the concentration function} of $\xi$ is defined as
$$\concf(\xi,\alpha)=\sup\limits_{\lambda\in\R}\P\bigl\{|\xi-\lambda|\le\alpha\bigr\},\;\;\alpha\ge 0.$$
The notion of the concentration function was introduced by P.~Levy \cite{Levy}
in context of studying distributions of sums of random variables.
Note that, for a random variable $\xi$ with zero median satisfying $\Exp|\xi|^{p}\ge m$, $\Exp|\xi|^q\le M$
for some $0<p<q$, $m,M>0$, necessarily $\concf(\xi,\alpha)\le 1-\beta$ for some $\alpha,\beta>0$ depending
only on $p,q,m,M$. At the same time, the condition $\concf(\xi,\alpha)\le 1-\beta$ for some $\alpha,\beta>0$
does not imply any upper bounds on positive moments of $\xi$.

The main result of our paper is the following theorem:

\begin{theor}\label{ssv theor nonsym}
For any real $\beta>0$ and $\delta>1$ there are $u,v>0$ and $N_0\in\N$ depending only on $\beta$ and $\delta$
with the following property:
Let $N,n\in\N$ satisfy $N\ge \max(N_0,\delta n)$; $A=(a_{ij})$ be an
$N\times n$ random matrix with i.i.d.\ entries, such that for some $\alpha>0$ the concentration function of the entries satisfies
\begin{equation}\label{the condition nonsym}
\concf(a_{11},\alpha)\le 1-\beta.
\end{equation}
Then for any non-random $N\times n$ matrix $B$ we have
\begin{equation}\label{the conclusion nonsym}
\P\bigl\{s_n(A+B)\le \alpha u\sqrt{N}\bigr\}\le \exp(-vN).
\end{equation}
\end{theor}

Adding a non-random component $B$ in the theorem does not increase complexity of the proof;
on the other hand, it demonstrates ``shift-invariance'' of the lower estimate.
Note that the problem of estimating the smallest singular value
of non-random shifts of {\it square} matrices is important
in analysis of algorithms
\cite{SST}, \cite{ST}, \cite{TV_MC}, \cite{TV_STOC}.

It is easy to see that a restriction of type \eqref{the condition nonsym}
is {\it necessary} for \eqref{the conclusion nonsym} to hold.
Indeed, suppose that for some $N\times n$ matrix $A$ with i.i.d.\ entries and some numbers
$u,v,\alpha>0$, \eqref{the conclusion nonsym} is true whenever $B=\lambda I$, $\lambda\in\R$.
Then, obviously,
$$\P\Bigl\{\sum\limits_{i=1}^N (a_{i1}-\lambda)^2\le \alpha^2u^2N\Bigr\}\le \exp(-vN),\;\;\lambda\in\R,$$
implying
$\concf(a_{11},\alpha u)=\sup\limits_{\lambda\in\R}\P\bigl\{|a_{11}-\lambda|\le \alpha u\bigr\}\le \exp(-v)$.

Our proof of Theorem~\ref{ssv theor nonsym} is based on two key elements: on a modification of a
standard $\varepsilon$-net argument for matrices (Proposition~\ref{gen spl proc})
and on estimates of the distance between
a random vector and a fixed linear subspace that follow from a result of
\cite{RV_SMALLBALL} (Theorem~\ref{RV conc lemma} and Corollary~\ref{RV extension} of our paper).
Our method is similar in many aspects to the approach developed in \cite{LPRT} and later in \cite{RV_RECT}, \cite{RV_SQUARE}.
In particular, as in the mentioned papers, we decompose the unit sphere $S^{n-1}$ into 
several subsets which are studied separately from one another.
On the other hand, our modification of the $\varepsilon$-net argument and its technical realization
in regard to splitting a random matrix into ``regular'' and ``non-regular'' parts
are apparently new.

We will discuss the main idea of the proof more concretely and in more detail
at the end of the next section, after we define notation and state the modified $\varepsilon$-net argument.

\section{Preliminaries}\label{prelim sec}

Throughout the text, $(\Omega,\Sigma,\P)$ denotes a probability space. Let $N\in \N$.
We say that a function $X:\Omega\to \R^N$ is a random vector in $\R^N$ if the preimage under $X$ of every Borel subset of $\R^N$
is $\P$-measurable. For any non-negative integer $n\le N$, let $\Gras(n,N)$ be the Grassmannian --- the set of all $n$-dimensional
subspaces of $\R^N$, equipped with the unique normalized rotation-invariant Borel measure (the Haar measure).
A function $V:\Omega\to \Gras(n,N)$ is an $n$-dimensional random subspace of $\R^N$
if the preimage under $V$ of every measurable subset of $\Gras(n,N)$ is $\P$-measurable.
It will be convenient for us to extend the last definition by allowing the subspace to have a
``variable'' dimension: Consider the disjoint union $\bigsqcup_{k=0}^n \Gras(k,N)$, with the measure
induced by the Haar measures on each $\Gras(k,N)$, $k=0,1,\dots,n$.
Then $V:\Omega\to \bigsqcup_{k=0}^n \Gras(k,N)$ is a random subspace of
$\R^N$ of dimension at most $n$ if the preimage under $V$ of every measurable subset of $\bigsqcup_{k=0}^n \Gras(k,N)$
is $\P$-measurable (an example of such a subspace in the text is $\thesubspace{A}{B}{H}{E}$
defined at the end of the section). We say that the random subspace $V$ and a random vector $X$ in $\R^N$ are independent
if for every Borel subset $K\subset\R^N$ and a measurable subset $K'\subset \bigsqcup_{k=0}^n \Gras(k,N)$ we have
$$\P\{X\in K\mbox{ and }V\in K'\}=\P\{X\in K\}\P\{V\in K'\}.$$

Given a vector $x\in\R^m$, by $\|x\|$ we denote the standard Euclidean norm and
by $\|x\|_\infty$ ~--- the $\ell_\infty^m$-norm of $x$.
By $S^{m-1}$ (respectively, $B_2^m$) we denote the Euclidean unit
sphere (respectively, the closed unit ball) in $\R^m$.
Further, for a set $K\subset\R^m$, $\dist(x,K)=\inf \{\|x-y\| : y\in K\}$ denotes  the Euclidean distance
between $x$ and $K$.
We use the same notation for the distance between two subsets of $\R^m$.

We will sometimes use the standard identification of $N\times n$ matrices
and linear operators from $\R^n$ to $\R^N$. In particular, for an
$N\times n$ matrix $D$ by $\|D\|$ we mean the operator norm of $D$
treated as the linear operator $D:\ell_2^n \to \ell_2^N$.
For a set $K\subset \R^n$, $D(K)$ is the image of $K$ in $\R^N$ under the action of $D$.
For an $N\times n$ matrix $D$, $\col_j(D)$ is the $j$-th column of $D$ and 
$\spn D$ is the linear span of columns of $D$ in $\R^N$.
The $N\times n$ matrix of ones is denoted by $\mofones$.
For a linear subspace $E\subset\R^n$, $E^\perp$
is the orthogonal complement of $E$ in $\R^n$ and $\Proj_E:\R^n\to\R^n$ is the orgothogonal projection onto $E$.
In the special case when $E$ is the linear span of a subset $\{e_j\}_{j\in J}$
($J\subset\{1,2,\dots,n\}$) of the standard unit basis in $\R^n$,
we will often write $x\chi_J$ in place of $\Proj_E(x)$.

In the paper, we define many universal constants and functions that are frequently referred to later in text.
For convenience, we add to the name of every such constant or function a subscript indicating
the statement where is was defined. For example,
$C_{\ref{net sparsification lemma}}$ is the universal constant from Lemma~\ref{net sparsification lemma}, etc.

Let  $K$ be a subset of $\R^n$ and let $\varepsilon>0$.
A subset $\Net\subset K$ is called an {\it $\varepsilon$-net} for $K$
if for any $y\in K$ there is $y'\in\Net$ with $\|y-y'\|\le\varepsilon$.
We will use a well-known fact that any subset $K\subset B_2^n$ admits an $\varepsilon$-net $\Net$ for $K$
with cardinality $|\Net|\le (3/\varepsilon)^n$.

Given an $\varepsilon$-net $\Net$ for $S^{n-1}$, the matrix $A+B$ from
Theorem~\ref{ssv theor nonsym} 
trivially satisfies $s_n(A+B)\ge
\min\limits_{y'\in\Net}\|Ay'+By'\|-\varepsilon\|A+B\|$. 
This standard $\varepsilon$-net argument is not applicable in our setting as $A+B$
may have a very large norm with a large probability.
A modification of the method in such a way that $\|A+B\|$ does not participate in the estimate for $s_n(A+B)$
is an important element of our proof. In this section we provide a ``non-probabilistic'' form of the argument.
Given a non-random $N\times n$ matrix $D$, we shall represent it
as a sum of two matrices $\regm$ and $\irregm$;
then we are able to estimate $s_n(D)$ from below in terms
of the norm $\|\regm\|$ of the ``regular part'' of
the matrix $D$ and distances between
certain vectors and subspaces in $\R^N$ (determined by matrices $\regm$ and $\irregm$).
We start with a simpler version of the argument:

\begin{lemma}\label{spl proc}
Let $N,n\in\N$, $h,\varepsilon>0$ and let $\regm,\irregm,D$ be $N\times n$ (non-random) matrices with $D=\regm+\irregm$.
Further, let $\Net$ be an $\varepsilon$-net on $S^{n-1}$ such that for any $y'\in\Net$ we have
$$
\dist\bigl(\regm y',\spn \irregm\bigr)\ge h.
$$
Then
$$s_n(D)\ge\inf\limits_{y\in S^{n-1}}\dist\bigl(\regm y,\spn \irregm\bigr)\ge h-\varepsilon\|\regm\|.$$
\end{lemma}
\begin{proof}
Choose any $y\in S^{n-1}$ and $y'\in\Net$ such that $\|y-y'\|\le\varepsilon$. Then
$$
\|Dy\|=\bigl\|\regm y+\irregm y\bigr\|
\ge \dist\bigl(\regm y, \spn \irregm\bigr)
\ge \dist\bigl(\regm y', \spn \irregm\bigr)-\varepsilon\|\regm\|
\ge h-\varepsilon\|\regm\|.
$$
By taking the infimum over all $y\in S^{n-1}$, we obtain the result.
\end{proof}

Note that Lemma~\ref{spl proc} cannot be used to handle matrices with the aspect ratio less than $2$.
Indeed, the lower estimate $s_n(D)\ge\inf\limits_{y\in S^{n-1}}\dist\bigl(\regm y,\spn \irregm\bigr)$
is non-trivial only if $\spn \regm\cap\spn \irregm=0$, which is not true when $N<2n$ and both $\regm$ and $\irregm$ have full rank.
The following strengthening of Lemma~\ref{spl proc} resolves the problem:
\begin{prop}\label{gen spl proc}
Let $N,n\in\N$, $S\subset S^{n-1}$ and let $\regm,\irregm,D$ be $N\times n$ (non-random) matrices with $D=\regm+\irregm$.
Further, suppose that the numbers $h,\varepsilon>0$, a subset $\Net\subset\R^n$ and a collection of linear subspaces
$\{E_{y'}\subset\R^n:\,y'\in\Net\}$ satisfy the following conditions:
$1)$ $y'\in E_{y'}$ for all $y'\in\Net$; $2)$ for any $y'\in\Net$ we have
\begin{equation}\label{gen spl cond 1}
\dist\bigl(\regm y',D(E_{y'}^\perp)+\irregm(E_{y'})\bigr)\ge h;
\end{equation}
and $3)$ for any $y\in S$ there is $y'\in\Net$ such that
$$\|\Proj_{E_{y'}}(y)-y'\|\le\varepsilon.$$
Then
$$\inf\limits_{y\in S}\|Dy\|\ge h-\varepsilon\|\regm\|.$$
\end{prop}
\begin{proof}
Take any $y\in S$ and let $y'\in\Net$ be such that $\|\Proj_{E_{y'}}(y)-y'\|\le\varepsilon$. Then
\begin{align*}
\|Dy\|&=\bigl\|\regm(\Proj_{E_{y'}}(y))+\bigl(D(\Proj_{E_{y'}^\perp}(y))+\irregm(\Proj_{E_{y'}}(y))\bigr)\bigr\|\\
&\ge \dist\bigl(\regm(\Proj_{E_{y'}}(y)), D(E_{y'}^\perp)+\irregm(E_{y'})\bigr)\\
&\ge \dist\bigl(\regm y', D(E_{y'}^\perp)+\irregm(E_{y'})\bigr)-\varepsilon\|\regm\|\\
&\ge h-\varepsilon\|\regm\|.
\end{align*}
Taking the infimum over $S$, we get the result.
\end{proof}

To apply Proposition~\ref{gen spl proc} we need an estimate for the distance between
a random vector in $\R^N$ with independent coordinates and a fixed linear subspace.
For any random vector $X$ in $\R^N$ define the concentration function of $X$ by
$$\concf(X,h)=\sup\limits_{\lambda\in\R^N}\P\bigl\{\|X-\lambda\|\le h\bigr\},\;\;h\ge 0.$$
Note that for $N=1$ the above definition is consistent with that given in the introduction.
The following result is proved by M.~Rudelson and R.~Vershynin in \cite{RV_SMALLBALL}:
\begin{theor}[\cite{RV_SMALLBALL}]\label{RV conc lemma}
Let $X=(X_1,X_2,\dots,X_m)$ be a random vector in $\R^m$ with independent coordinates such that
$$\concf(X_i,h)\le \eta,\;\;i=1,2,\dots,m$$
for some $h>0,\eta\in(0,1)$. Then for any $d\in\{1,2,\dots,m\}$ and any $d$-dimensional non-random subspace $E\subset\R^m$
$$\concf(\Proj_E X,h\sqrt{d})\le (C_{\ref{RV conc lemma}}\eta)^d,$$
where $C_{\ref{RV conc lemma}}>0$ is a (sufficiently large) universal constant.
\end{theor}
This theorem gives a nontrivial estimate for concentration only for $\eta$ sufficiently close to zero.
Below, we provide an elementary estension of this result
covering the case of ``more concentrated'' coordinates. First, let us recall a theorem of B.~Rogozin:
\begin{theor}[{\cite{Rogozin}}]\label{rogozin lemma}
Let $k\in\N$, $\xi_1,\xi_2,\dots,\xi_k$ be independent random variables and let $h_1$, $h_2,\dots$, $h_k>0$ be some
real numbers. Then for any $h\ge\max\limits_{j=1,2,\dots,k}h_j$,
$$\concf\Bigl(\sum\limits_{j=1}^k\xi_j,h\Bigr)\le C_{\ref{rogozin lemma}}h\Bigl(\sum\limits_{j=1}^k\bigl(1-\concf(\xi_j,h_j)\bigr)h_j^2\Bigr)^{-1/2},$$
where $C_{\ref{rogozin lemma}}>0$ is a universal constant.
\end{theor}
Now, an easy application of Theorems~\ref{RV conc lemma} and~\ref{rogozin lemma} gives
\begin{cor}\label{RV extension}
Let $X=(X_1,X_2,\dots,X_m)$ be a random vector with independent coordinates such that
$$\concf(X_i,h)\le 1-\tau,\;\;i=1,2,\dots,m$$
for some $h>0,\tau\in(0,1)$. Then for any $d\in\{1,2,\dots,m\}$, $\ell\in\N$ and any $d$-dimensional non-random subspace $E\subset\R^m$
the concentration function of $\Proj_E X$ satisfies
$$\concf(\Proj_E X,h\sqrt{d}/\ell)\le \bigl(C_{\ref{RV conc lemma}}C_{\ref{rogozin lemma}}/\sqrt{\ell \tau}\bigr)^{d/\ell}.$$
\end{cor}
\begin{proof}
Let $X^1,X^2,\dots,X^\ell$ be independent copies of $X$ and $S=(S_1,S_2,\dots,S_m)=\sum\limits_{j=1}^\ell X^j$. Then, in view of the
condition on coordinates of $X$ and Theorem~\ref{rogozin lemma}, we obtain
$$\concf(S_i,h)\le C_{\ref{rogozin lemma}}\Bigl(\ell\bigl(1-\concf(X_i,h)\bigr)\Bigr)^{-1/2}
\le \frac{C_{\ref{rogozin lemma}}}{\sqrt{\ell \tau}},\;\;i=1,2,\dots,m.$$
Then Theorem~\ref{RV conc lemma} gives
$$\concf(\Proj_E S,h\sqrt{d})\le \bigl(C_{\ref{RV conc lemma}}C_{\ref{rogozin lemma}}/\sqrt{\ell \tau}\bigr)^d,$$
and via the definition of $S$ we get the statement.
\end{proof}
\begin{rem}
Note that for any non-zero $\tau$ we can choose $\ell\in\N$ such that
the upper estimate for the concentration function provided by Corollary~\ref{RV extension} is non-trivial
(strictly less than $1$).
In fact, a slightly weaker version of Corollary~\ref{RV extension} still sufficient for our purposes could be proved using
the original result of P.~Levy from \cite{Levy} instead of Theorem~\ref{rogozin lemma}.
\end{rem}

As an immediate application of Corollary~\ref{RV extension}, we prove a statement about {\it peaky} vectors.
We call a vector $y\in S^{n-1}$ {\it $\theta$-peaky} for some $\theta>0$ if $\|y\|_\infty\ge\theta$.
The set of all $\theta$-peaky unit vectors in $\R^n$ shall be denoted by $S^{n-1}_p(\theta)$.

\begin{prop}[Peaky vectors]\label{peaky lemma}
Let $\delta>1$ and let $n,N\in\N$ satisfy $N\ge \delta n$. Further, assume we are given $\theta,\gamma>0$
and let $U=(u_{ij})$ be an $N\times n$ random matrix with independent entries (not necessarily identically distributed),
each entry $u_{ij}$ satisfying
$$\concf(u_{ij},1)\le 1-\gamma.$$
Then
$$\P\Bigl\{\inf\limits_{y\in S_p^{n-1}(\theta)}\|Uy\|\le h_{\ref{peaky lemma}}\theta
\sqrt{N}\Bigr\}\le n\exp(-w_{\ref{peaky lemma}}N),$$
where the $h_{\ref{peaky lemma}},w_{\ref{peaky lemma}}>0$ depend only on $\gamma$ and $\delta$.
\end{prop}
\begin{proof}
By Corollary~\ref{RV extension}, for $d=N-n+1$, any $\ell\in\N$ and any fixed $(n-1)$-dimensional subspace $F\subset\R^N$
we have
\begin{align*}
\P\bigl\{\dist(\col_j(U),F)\le \sqrt{d}/\ell\bigr\}
&\le\concf\bigl(\Proj_{F^\perp} (\col_j(U)),\sqrt{d}/\ell\bigr)\\
&\le \bigl(C_{\ref{RV conc lemma}}C_{\ref{rogozin lemma}}/\sqrt{\ell \gamma}\bigr)^{d/\ell},\;\;j=1,2,\dots,n.
\end{align*}
Take $\ell:=\lceil 4C_{\ref{RV conc lemma}}^2 C_{\ref{rogozin lemma}}^2/\gamma\rceil $.
Since for each $j=1,2,\dots,n$, $\col_j(U)$ is independent from the span of the other columns of $U$,
from the above estimate we obtain
$$\P\bigl\{\dist\bigl(\col_j(U), \spn \{\col_k(U)\}_{k\neq j}\bigr)\le h\sqrt{d}\bigr\}
\le \exp\bigl(-wd\bigr),\;\;j=1,2,\dots,n$$
for some $h,w>0$ depending only on $\gamma$. Let
$$\mathcal E=\bigl\{\omega\in\Omega:\,
\dist\bigl(\col_j(U(\omega)), \spn \{\col_k(U(\omega))\}_{k\neq j}\bigr)> h\sqrt{d}\mbox{ for all }j=1,2,\dots,n\bigr\}.$$
Then $\P(\mathcal E)\ge 1-n\exp(-wd)$.
Take arbitrary $\omega\in\mathcal E$. For any $y=(y_1,y_2,\dots,y_n)$ in $S_p^{n-1}(\theta)$ there is $j=j(y)$ such that
$|y_j|\ge\theta$, hence
\begin{align*}
\|U(\omega)y\|
&=\|U(\omega)(y_je_j)+U(\omega)(y-y_je_j)\|\\
&\ge\theta\dist\bigl(\col_j(U(\omega)), \spn \{\col_k(U(\omega))\}_{k\neq j}\bigr)\\
&>h\theta \sqrt{d}.
\end{align*}
Thus,
$$\P\Bigl\{\inf\limits_{y\in S_p^{n-1}(\theta)}\|Uy\|\le h\theta \sqrt{d}\Bigr\}\le n\exp(-wd),$$
and the statement follows.
\end{proof}

Next, we introduce two notions important for us that will be used throughout the rest of the text.
For any number $s\in\R$ and any Borel subset $H\subset\R$, define {\it the $H$-part of $s$} as
$$\trunk{s}{H}=\begin{cases}s,&\mbox{if $s\in H$},\\0,&\mbox{otherwise}.\end{cases}$$
The ``complementary'' $\R\backslash H$-part of $s$
will be denoted by $\trunk{s}{\overline H}$. Obviously, $s=\trunk{s}{H}+\trunk{s}{\overline H}$.
The name and the notation resemble the {\it positive} and {\it negative} part of a real number;
in fact $s_+=\trunk{s}{H}$ for $H=[0,\infty)$.
For a real-valued random variable $\xi$ we define the $H$-part of $\xi$ pointwise:
$\trunk{\xi}{H}(\omega)= \trunk{\xi(\omega)}{H}$ for all $\omega\in\Omega$.
When a variable has a subscript, we will use parentheses to separate the subscript from
the $H$-part notation, for example $\trunk{(\xi_1)}{H}$ is the $H$-part of a random variable $\xi_1$.
Given a matrix $A=(a_{ij})$, its $H$-part $\trunk{A}{H}$ is defined 
entry-wise, i.e.\ $(\trunk{A}{H})_{ij}=\trunk{(a_{ij})}{H}$ for all admissible $i,j$.

For any $N\times n$ matrices $M,M'$ (whether random or not),
a Borel set $H\subset\R$ and a linear subspace $E\subset\R^n$ let
$$\thesubspace{M}{M'}{H}{E}:=(M+M')(E^\perp)+(\trunk{M}{\overline H}+M')(E).$$
Note that $\thesubspace{M}{M'}{H}{E}$ is a linear subspace of $\R^N$ of dimension at most $n$. When the
matrices $M$, $M'$ are clear from the context, we shall write $\thesubspaceshort{H}{E}$ in place of
$\thesubspace{M}{M'}{H}{E}$.
When one or both matrices $M,M'$ are random, $\thesubspace{M}{M'}{H}{E}$ is a {\it random subspace in $\R^N$}
of dimension at most $n$,
that can be formally viewed as a function from $\Omega$ to the disjoint union of Grassmannians $\bigsqcup_{k=0}^n\Gras(k,N)$,
$k=0,1,\dots,n$ (see the beginning of this section).

\bigskip

Let us conclude the section by describing the main idea of the proof of Theorem~\ref{ssv theor nonsym}.
Let $S$ be a subset of $S^{n-1}$.
As we already noted before, the main obstacle in using the standard $\varepsilon$-net argument to get a lower estimate
for $\inf\limits_{y\in S}\|Ay+By\|$ is the need to control the norm of the matrix $A+B$ which is not possible unless we impose strong
restrictions on its entries. Proposition~\ref{gen spl proc} provides a workaround:
we represent $A+B$ as a sum of two random matrices, ``regular'' and ``irregular'', satisfying certain conditions, so that
the lower bound for $\inf\limits_{y\in S}\|Ay+By\|$ involves the norm of only the ``regular'' matrix.
The splitting shall be defined with help of the above concept of $H$-part. Namely, for some specially chosen
$\lambda\in\R$ and $H\subset\R$ we define the ``regular'' part as $\trunk{(A-\lambda\mofones)}{H}$
and the ``irregular'' as $A+B-\trunk{(A-\lambda\mofones)}{H}$ (which is identical to
$\trunk{(A-\lambda\mofones)}{\overline H}+B+\lambda\mofones$). The set $H$ shall be bounded which
implies boundedness of the entries of $\trunk{(A-\lambda\mofones)}{H}$.
This, together with the appropriately chosen ``shift'' $\lambda$, allows us to easily control $\|\trunk{(A-\lambda\mofones)}{H}\|$
from above. We will define $H$ as the union of two specially constructed closed intervals on $\R$.
The choice of $H$ depends on the set $S$ and may depend on the characteristics
of the distribution of the entries of $A$ (we leave this problem for the last section).

The crucial property that our set $H$ shall satisfy is: letting $\tilde A=A-\lambda\mofones$ and $\tilde B=B+\lambda\mofones$,
for certain finite subset of vectors $\Net\subset \R^n$
and a collection of linear subspaces $\{E_{y'}\subset\R^n\}_{y'\in\Net}$ (see Proposition~\ref{gen spl proc}) we have
$$\inf\limits_{y'\in\Net}\dist\bigl(\trunk{\tilde A}{H}y',\thesubspace{\tilde A}{\tilde B}{H}{E_{y'}}\bigr)\gtrsim\sqrt{N}$$
with a large probability. This restriction on $H$ naturally corresponds to the condition \eqref{gen spl cond 1}
in Proposition~\ref{gen spl proc}.
In practice we shall verify this property of $H$ by proving that for every vector $y\in B_2^n$ satifying certain upper bounds
on $\|y\|_\infty$ and lower bounds on $\|y\|$ and for $E=\spn\{e_j\}_{j\in\supp y}$, the distance
$\dist\bigl(\trunk{\tilde A}{H}y,\thesubspace{\tilde A}{\tilde B}{H}{E}\bigr)$ is large with an overwhelming probability.
This condition demands a ``rich'' structure from $\trunk{\tilde A}{H}$; consequently, the set $H$
cannot be very small in diameter. On the other hand, the ``upper'' restrictions on $H$ are dictated by the necessity to
control the norm of $\trunk{\tilde A}{H}$. Thus, we have to find a balance between the two requirements.

In order to estimate the distance between the random vector $\trunk{\tilde A}{H}y$
and the random subspace $\thesubspace{\tilde A}{\tilde B}{H}{E}$, we will use Corollary~\ref{RV extension}.
However, since in general $\thesubspace{\tilde A}{\tilde B}{H}{E}$ is {\it dependent}
(in probabilistic sense) on $\trunk{\tilde A}{H}y$, an immediate application of the corollary is not possible;
instead, we will combine it with a conditioning argument, which is presented in the next section.

\section{The distribution of $\dist\bigl(\trunk{A}{H}y,\thesubspace{A}{B}{H}{E}\bigr)$}\label{sec dec}

Assume that we are given $\delta>1$, $N,n\in\N$ with $N\ge\delta n$, a random $N\times n$ matrix $A$
with i.i.d.\ entries, a non-random $N\times n$ matrix $B$ and a Borel subset $H\subset \R$ with $\P\{a_{11}\in H\}>0$.
The purpose of this section is to study the distribution
of the distance between a random vector $\trunk{A}{H}y$ and the random subspace
$\thesubspace{A}{B}{H}{E}=(A+B)(E^\perp)+(\trunk{A}{\overline H}+B)(E)$,
where $E=\spn\{e_j\}_{j\in\supp y}$. We give {\it sufficient} conditions on $A$, $H$ and $y$ which
guarantee that $\dist\bigl(\trunk{A}{H}y,\thesubspace{A}{B}{H}{E}\bigr)$ is large
with a large probability (Proposition~\ref{distance estimate}).
Note that generally $\trunk{A}{H}y$ and $\thesubspace{A}{B}{H}{E}$ are {\it dependent}.
In order to overcome this problem, we apply a decoupling argument.

We adopt the following notation:
For any subset $W\subset\{1,2,\dots,N\}\times\{1,2,\dots,n\}$ let
$$\Omega_W=\bigl\{\omega\in\Omega:\,a_{ij}(\omega)\in H\mbox{ for all }(i,j)\in W
\mbox{ and }a_{ij}(\omega)\in {\overline H}\mbox{ for all }(i,j)\notin W\bigr\}.$$

Given an event $\mathcal E\subset\Omega$ with $\P(\mathcal E)>0$,
we denote by $(\mathcal E,\Sigma_{\mathcal E},\P_{\mathcal E})$
the probability space where the $\sigma$-algebra $\Sigma_{\mathcal E}$ of
subsets of $\mathcal E$ is naturally induced by the $\sigma$-algebra $\Sigma$ on $\Omega$,
and $\P_{\mathcal E}$ is defined by $\P_{\mathcal E}(K)=\P(\mathcal E)^{-1}\P(K)$ ($K\in \Sigma_{\mathcal E}$).

\begin{lemma}[Conditional independence]\label{decoupl lemma}
Let $A$, $B$ and $H$ be as above, $y\in\R^n$, $E=\spn\{e_j\}_{j\in\supp y}$ and let
$W\subset\{1,2,\dots,N\}\times\{1,2,\dots,n\}$ be such that $\P(\Omega_W)>0$.
Then the random vector $\trunk{A}{H}y$ in $\R^N$ and
the random subspace $\thesubspace{A}{B}{H}{E}\subset\R^N$
are conditionally independent given event $\Omega_W$. Moreover, the coordinates of $\trunk{A}{H}y$ are
conditionally independent given $\Omega_W$.
\end{lemma}
\begin{proof}
If $\P\{a_{11}\in \overline H\}=0$ then the assumption $\P(\Omega_W)>0$ necessarily implies that
$W=\{1,2,\dots,N\}\times\{1,2,\dots,n\}$ and $\Omega_W=\Omega$ (up to a set of $\P$-measure zero).
At the same time, in this case
$\thesubspace{A}{B}{H}{E}=(A+B)(E^\perp)+B(E)$ a.s., hence $\thesubspace{A}{B}{H}{E}$ and $\trunk{A}{H}y$
are independent on $\Omega=\Omega_W$, and we get the statement.

Now, assume that $\P\{a_{11}\in \overline H\}\neq 0$.
It is enough to check that
\begin{align}
&\mbox{the random variables}\;\;\trunk{(a_{ij})}{H},\;\trunk{(a_{ij})}{\overline H}\;\;(1\le i\le N,\;1\le j\le n)\nonumber\\
&\mbox{are jointly conditionally independent given}\;\;\Omega_W.\label{cond ind aux 8362}
\end{align}
Note that for all $\omega\in\Omega_W$ we have
$\trunk{(a_{ij})}{\overline H}(\omega)=0$ for $(i,j)\in W$ and $\trunk{(a_{ij})}{H}(\omega)=0$ for $(i,j)\notin W$.
Hence, to verify \eqref{cond ind aux 8362} it is sufficient to prove that $nN$ variables
$$\trunk{(a_{ij})}{H},\;(i,j)\in W;\;\;\trunk{(a_{ij})}{\overline H},\;(i,j)\notin W$$
are jointly conditionally independent given $\Omega_W$.
But for $(i,j)\in W$ the $H$-part of $a_{ij}$ satisfies $\trunk{(a_{ij})}{H}=a_{ij}$ {\it everywhere on $\Omega_W$}
and, similarly, for $(i,j)\notin W$, we have $\trunk{(a_{ij})}{\overline H}=a_{ij}$ everywhere on $\Omega_W$.
Hence, once we verify conditional independence for $a_{ij}$ ($1\le i\le N$, $1\le j\le n$) given $\Omega_W$,
then we immediately get \eqref{cond ind aux 8362}.
For any Borel subsets $K_{ij}\subset\R$ ($1\le i\le N$, $1\le j\le n$) we have
\begin{align*}
&\prod\limits_{(i,j)}\P_{\Omega_W}\bigl\{a_{ij}\in K_{ij}\bigr\}\\
&=\prod\limits_{(i,j)}\bigl(\P(\Omega_W)^{-1}\P\bigl\{\omega\in\Omega_W:\,a_{ij}(\omega)\in K_{ij}\bigr\}\bigr)\\
&=\prod\limits_{(i,j)\in W}\bigl(\P\bigl\{a_{ij}\in H\bigr\}^{-1}\P\bigl\{a_{ij}\in H\cap K_{ij}\bigr\}\bigr)\,
\prod\limits_{(i,j)\notin W}\bigl(\P\bigl\{a_{ij}\in \overline{H}\bigr\}^{-1}\P\bigl\{a_{ij}\in\overline{H}\cap K_{ij}\bigr\}\bigr)\\
&=\P(\Omega_W)^{-1}\P\bigl\{a_{ij}\in H\cap K_{ij}\mbox{ for all $(i,j)\in W$ and }a_{ij}\in \overline{H}\cap K_{ij}
\mbox{ for all $(i,j)\notin W$}\bigr\}\\
&=\P_{\Omega_W}\bigl\{a_{ij}\in K_{ij}:\,1\le i\le N,\,1\le j\le n\bigr\},
\end{align*}
so $a_{ij}$ ($1\le i\le N$, $1\le j\le n$) are conditionally independent given $\Omega_W$.
\end{proof}

Lemma~\ref{decoupl lemma} shows that
Corollary~\ref{RV extension} can be applied to $\trunk{A}{H}y$ and the subspace $\thesubspace{A}{B}{H}{E}$
``inside'' each $\Omega_W$. Hence, to give a satisfactory lower estimate for $\dist\bigl(\trunk{A}{H}y,\thesubspace{A}{B}{H}{E}\bigr)$
on entire $\Omega$, it is enough to verify that there is a subset $M\subset 2^{\{1,2,\dots,N\}\times\{1,2,\dots,n\}}$ such that
the $\P$-measure of the union of $\Omega_W$'s ($W\in M$) is close to $1$ and 
for each $W\in M$, the restriction of the vector $\trunk{A}{H}y$ to $\Omega_W$
has sufficiently ``spread'' coordinates.
Of course, such a set $M$ may exist only under certain assumptions on $A$, $H$ and $y$.
In Lemma~\ref{conc in matrix}, we formulate those assumptions using 
random variables that agree on a part of the probability space and are independent when
restricted to the other part of $\Omega$. Let us remark that, whereas the use of such variables has some
advantages (in our opinion), it should not be regarded as a necessary ingredient of the proof.

Let $\xi,\xi'$ be two random variables such that
$\P\{\xi\in H\}>0$. We say that 
{\it $\xi,\xi'$ are conditionally i.i.d.\ given event $\{\omega\in\Omega:\,\xi(\omega)\in H\}$ and identical on
$\{\omega\in\Omega:\,\xi(\omega)\in\overline{H}\}$} if the following is true:
setting $\mathcal E=\{\omega\in\Omega:\,\xi(\omega)\in H\}$, the restrictions
of $\xi,\xi'$ to the probability space $(\mathcal E,\Sigma_{\mathcal E},\P_{\mathcal E})$ are i.i.d.\ and
$\xi(\omega)=\xi'(\omega)$ for $\omega\in \Omega\setminus\mathcal E$. The definition implies that
$\xi'$ has the same individual distribution (on $\Omega$) as $\xi$ and for any Borel subsets $K,K'\subset\R$
$$\P\bigl\{(\xi,\xi')\in K\times K'\bigr\}=
\frac{\P\{\xi\in H\cap K\}\P\{\xi\in H\cap K'\}}{\P\{\xi\in H\}}+\P\{\xi\in \overline{H}\cap K\cap K'\};$$
in particular, $\P\{(\xi,\xi')\in H\times \overline{H}\}=\P\{(\xi,\xi')\in \overline{H}\times H\}=0$.
Note that
$\trunk{\xi}{\overline H}$ and $\trunk{\xi'}{\overline H}$ are equal a.s.\ on $\Omega$.
It is a trivial observation that $\trunk{\xi}{H}-\trunk{\xi'}{H}$ is symmetrically distributed.

For any event $\mathcal E\subset\Omega$ with $\P(\mathcal E)>0$ and any random variable $\xi$ on $\Omega$,
let $\concf_{\mathcal E}(\xi,\cdot)$ be the concentration function of the restriction of $\xi$ to the probability space
$(\mathcal E,\Sigma_{\mathcal E},\P_{\mathcal E})$.

\begin{lemma}\label{conc in matrix}
Let $H$ be a Borel subset of $\R$; $N\ge\delta n$ for some
$\delta>1$ and let $A=(a_{ij})$ be an $N\times n$ random matrix with i.i.d.\ entries and $\P\{a_{11}\in H\}>0$.
Further, let $A'=(a_{ij}')$ be an $N\times n$ random matrix having the same distribution as $A$ such that
$2$-dimensional vectors $(a_{ij},a_{ij}')$ ($1\le i\le N$, $1\le j\le n$)
are i.i.d.\ and for any admissible $i$ and $j$ the variables
$a_{ij}$ and $a_{ij}'$ are conditionally i.i.d.\ given event $\{\omega\in\Omega:\,a_{ij}(\omega)\in H\}$ and identical on
$\{\omega\in\Omega:\,a_{ij}(\omega)\in\overline H\}$.
Let $y=(y_1,y_2,\dots,y_n)\in\R^n$ and $s>0$ be such that
\begin{equation}\label{new app aux 125}
\P\Bigl\{\Bigl|\sum\limits_{j=1}^n \bigl(\trunk{(a_{ij})}{H}-\trunk{(a_{ij}')}{H}\bigr) y_j\Bigr|>s\Bigr\}\ge \delta^{-1/4},\;\;i=1,2,\dots,N.
\end{equation}
Define $M$ as the collection of all subsets $W\subset \{1,2,\dots,N\}\times\{1,2,\dots,n\}$ satisfying
$$\P(\Omega_W)>0\;\;\mbox{and}\;\;
\Bigl|\Bigl\{i\in\{1,2,\dots,N\}:\,\concf_{\Omega_W}\Bigl(\sum\limits_{j=1}^n
\trunk{(a_{ij})}{H}y_j,\frac{s}{2}\Bigr)\le 1-\tau\Bigr\}\Bigr|\ge N\delta^{-1/2}$$
with $\tau=\frac{1}{2}\bigl(\delta^{-1/4}-\delta^{-1/3}\bigr)$.
Then
$$\P\Bigl(\bigcup_{W\in M}\Omega_W\Bigr)\ge 1-\exp(-w_{\ref{conc in matrix}} N),$$
where $w_{\ref{conc in matrix}}>0$ depends only on $\delta$.
\end{lemma}
\begin{proof}
For each $i=1,2,\dots,N$ and $J\subset\{1,2,\dots,n\}$ let
$$\Omega_J^i=\bigl\{\omega\in\Omega:\,a_{ij}(\omega)\in H\mbox{ for all }j\in J
\mbox{ and }a_{ij}(\omega)\in {\overline H}\mbox{ for all }j\notin J\bigr\},$$
and for $i=1,2,\dots,N$ define
$$L_i=\Bigl\{J\subset \{1,2,\dots,n\}:\,\P(\Omega_J^i)>0\mbox{ and }
\concf_{\Omega_J^i}\Bigl(\sum\limits_{j=1}^n \trunk{(a_{ij})}{H}y_j,\frac{s}{2}\Bigr)\le 1-\tau\Bigr\};\;
\mathcal E_i=\bigcup\limits_{J\in L_i}\Omega_J^i.$$
It is not difficult to see that the events $\mathcal E_i\subset\Omega$ ($i=1,2,\dots,N$)
are independent in view of independence of the entries of $A$.

Fix for a moment any $i\in\{1,2,\dots,N\}$. One can verify that
for any $j\in\{1,2,\dots,n\}$ and $J\subset\{1,2,\dots,n\}$ the variables $\trunk{(a_{ij})}{H}$
and $\trunk{(a_{ij}')}{H}$ are i.i.d.\ given event $\Omega_J^i$.
It follows that
\begin{equation}\label{cond indep 123}
\sum_{j=1}^n \trunk{(a_{ij})}{H}y_j\;\mbox{ and }\;\sum_{j=1}^n \trunk{(a_{ij}')}{H}y_j
\;\mbox{ are i.i.d.\ given $\Omega_J^i$, for all $J\subset\{1,2,\dots,n\}$}.
\end{equation}
Take any subset $J\subset\{1,2,\dots,n\}$ satisfying
\begin{equation}\label{new app aux 1351}
\P(\Omega_J^i)>0\mbox{ and }
\P_{\Omega_J^i}\Bigl\{\Bigl|\sum\limits_{j=1}^n \bigl(\trunk{(a_{ij})}{H}-\trunk{(a_{ij}')}{H}\bigr) y_j\Bigr|> s\Bigr\}\ge 2\tau.
\end{equation}
For all $\lambda\in\R$ we have, in view of \eqref{cond indep 123},
\begin{align*}
\P_{\Omega_J^i}&\Bigl\{\lambda-\frac{s}{2}\le\sum\limits_{j=1}^n \trunk{(a_{ij})}{H}y_j\le\lambda+\frac{s}{2}\Bigr\}^2\\
&=\P_{\Omega_J^i}\Bigl\{\lambda-\frac{s}{2}\le\sum\limits_{j=1}^n \trunk{(a_{ij})}{H}y_j\le\lambda+\frac{s}{2}\mbox{ and }
\lambda-\frac{s}{2}\le\sum\limits_{j=1}^n \trunk{(a_{ij}')}{H}y_j\le\lambda+\frac{s}{2}\Bigr\}\\
&\le\P_{\Omega_J^i}\Bigl\{\Bigl|\sum\limits_{j=1}^n \bigl(\trunk{(a_{ij})}{H}-\trunk{(a_{ij}')}{H}\bigr) y_j\Bigr|\le s\Bigr\}\\
&\le 1-2\tau,
\end{align*}
implying
$$\concf_{\Omega_J^i}\Bigl(\sum\limits_{j=1}^n \trunk{(a_{ij})}{H}y_j,\frac{s}{2}\Bigr)\le \sqrt{1-2\tau}\le 1-\tau.$$
Thus, any $J$ satisfying \eqref{new app aux 1351} belongs to $L_i$.
Clearly,
\begin{equation*}\label{new app aux 135}
\P\Bigl\{\Bigl|\sum\limits_{j=1}^n \bigl(\trunk{(a_{ij})}{H}-\trunk{(a_{ij}')}{H}\bigr) y_j\Bigr|
> s\Bigr\}=\sum\limits_{J}\P_{\Omega_J^i}\Bigl\{\Bigl|\sum\limits_{j=1}^n
\bigl(\trunk{(a_{ij})}{H}-\trunk{(a_{ij}')}{H}\bigr) y_j\Bigr|>s\Bigr\}\P(\Omega_J^i),
\end{equation*}
where the summation is taken over $J\subset\{1,2,\dots,n\}$ satisfying $\P(\Omega_J^i)>0$.
Hence, in view of \eqref{new app aux 125} and the above observations we get
\begin{align*}
\delta^{-1/4}&\le \sum\limits_{J}\P_{\Omega_J^i}\Bigl\{\Bigl|\sum\limits_{j=1}^n
\bigl(\trunk{(a_{ij})}{H}-\trunk{(a_{ij}')}{H}\bigr) y_j\Bigr|>s\Bigr\}\P(\Omega_J^i)\\
&\le \sum\limits_{J\in L_i}\P(\Omega_J^i)+2\tau\sum\limits_{J\notin L_i}\P(\Omega_J^i)\\
&\le2\tau+\P(\mathcal E_i),
\end{align*}
implying $\P(\mathcal E_i)\ge\delta^{-1/3}$.

We have shown that the events $\mathcal E_i$ ($i=1,2,\dots,N$) are independent and $\P(\mathcal E_i)\ge\delta^{-1/3}$ for each $i$.
Now, setting
$$\mathcal E=\bigl\{\omega\in\Omega:\,\bigl|\bigl\{i\in\{1,2,\dots,N\}:\,\omega\in\mathcal E_i\bigr\}\bigr|\ge N\delta^{-1/2}\bigr\},$$
we obtain by Bernstein's (or Hoeffding's) inequality
$\P(\mathcal E)\ge 1-\exp(-w_{\ref{conc in matrix}} N)$, where $w_{\ref{conc in matrix}}>0$
depends only on $\delta$. Take any $W\subset \{1,2,\dots,N\}\times\{1,2,\dots,n\}$ such that $\P(\mathcal E\cap \Omega_W)>0$.
Then, by the construction of $\mathcal E$, there is a subset $I\subset \{1,2,\dots,N\}$
of cardinality at least $N\delta^{-1/2}$ and sets $J_i\in L_i$ ($i\in I$)
such that $\Omega_W\subset\Omega_{J_i}^i$ for all $i\in I$. For every $i\in I$ by the definition of $L_i$ we have
$$\concf_{\Omega_{J_i}^i}\Bigl(\sum\limits_{j=1}^n \trunk{(a_{ij})}{H}y_j,\frac{s}{2}\Bigr)\le 1-\tau,$$
hence $W\subset M$. The argument implies
$\mathcal E\subset\bigcup_{W\in M}\Omega_W$ and the result follows.
\end{proof}

Next, we combine the result of Lemma~\ref{conc in matrix} with Corollary~\ref{RV extension}:

\begin{lemma}\label{wrap signum lemma}
Let $N,n,\delta$, $H$, $A,A'$, $y$ and $s$ be exactly as in Lemma~\ref{conc in matrix}
and $B$ be a non-random $N\times n$ matrix.
Then
\begin{align*}
\P\bigl\{\dist\bigl(\trunk{A}{H}y, \thesubspace{A}{B}{H}{E}\bigr)\le s h_{\ref{wrap signum lemma}}\sqrt{N}\bigr\}
\le 2\exp\bigl(-w_{\ref{wrap signum lemma}}N\bigr),
\end{align*}
where $E=\spn\{e_j\}_{j\in\supp y}$ and $h_{\ref{wrap signum lemma}}>0$, $w_{\ref{wrap signum lemma}}>0$ depend only on $\delta$.
\end{lemma}
\begin{proof}
Let $M$ and $\tau$ be defined as in Lemma~\ref{conc in matrix} and take any $W\in M$.
Let
$$m=\Bigl|\Bigl\{i\in\{1,2,\dots,N\}:\,\concf_{\Omega_W}\Bigl(\sum\limits_{j=1}^n \trunk{(a_{ij)}}{H}y_j,\frac{s}{2}\Bigr)\le 1-\tau\Bigr\}\Bigr|.$$
By the definition of $M$, we have $m\ge N\delta^{-1/2}\ge \sqrt{\delta}n$, hence, taking
$d=m-n$ and $\ell=4(C_{\ref{RV conc lemma}}C_{\ref{rogozin lemma}})^2/\tau$, by Corollary~\ref{RV extension},
for $\kappa=\delta^{-1/2}-\delta^{-1}$ and any fixed $n$-dimensional subspace $F\subset\R^N$ we obtain
$$\P_{\Omega_W}\Bigl\{\dist\bigl(\trunk{A}{H}y,F\bigr)
\le \frac{s}{2\ell}\sqrt{\kappa N}\Bigr\}\le 2^{-\kappa N/\ell}.$$
By Lemma~\ref{decoupl lemma}, the subspace $\thesubspace{A}{B}{H}{E}=(A+B)(E^\perp)+(\trunk{A}{\overline H}+B)(E)$
and the vector $\trunk{A}{H}y$ are conditionally independent given $\Omega_W$,
hence the above estimate immediately implies
$$\P_{\Omega_W}\Bigl\{\dist\bigl(\trunk{A}{H}y,\thesubspace{A}{B}{H}{E}\bigr)
\le \frac{s}{2\ell}\sqrt{\kappa N}\Bigr\}\le 2^{-\kappa N/\ell}.$$
Since the relation holds for all $W\in M$, in view of Lemma~\ref{conc in matrix} we obtain
\begin{align*}
\P\Bigl\{\dist\bigl(\trunk{A}{H}y,\thesubspace{A}{B}{H}{E}\bigr)
\le \frac{s}{2\ell}\sqrt{\kappa N}\Bigr\}
&\le 2^{-\kappa N/\ell}\,\P\Bigl(\bigcup_{W\in M}\Omega_W\Bigr)+1-\P\Bigl(\bigcup_{W\in M}\Omega_W\Bigr)\\
&\le 2^{-\kappa N/\ell}+\exp(-w_{\ref{conc in matrix}} N),
\end{align*}
and the result follows.
\end{proof}

Finally, we can prove the main result of the section:

\begin{prop}\label{distance estimate}
Let $\delta>1$, $n,N\in\N$, $N\ge\delta n$ and let $A=(a_{ij})$ be an $N\times n$ random matrix with i.i.d.\ entries
and $B$ be any non-random $N\times n$ matrix. Further, for some $d,r>0$ let
$H$ be a Borel subset of $\R$ such that
$H=H_1\cup H_2$ for disjoint Borel sets $H_1$, $H_2$ with
$\dist(H_1,H_2)\ge d$ and
$\min\bigl(\P\{a_{11}\in H_1\},\P\{a_{11}\in H_2\}\bigr)\ge r$.
For arbitrary $t>0$ define
$$h_{\ref{distance estimate}}=\frac{1-\delta^{-1/4}}{C_{\ref{rogozin lemma}}}\sqrt{\frac{r}{8}}td$$
and let $y\in\R^n$ be a vector satisfying
$\|y\|\ge t$, $\|y\|_\infty\le \frac{2h_{\ref{distance estimate}}}{d}$ and $E=\spn\{e_j\}_{j\in\supp y}$.
Then
$$\P\bigl\{\dist\bigl(\trunk{A}{H}y,\thesubspace{A}{B}{H}{E}\bigr)
\le h_{\ref{wrap signum lemma}}h_{\ref{distance estimate}}\sqrt{N}\bigr\}\le 2\exp(-w_{\ref{wrap signum lemma}} N).$$
\end{prop}
\begin{proof}
Let $A'=(a_{ij}')$ be an $N\times n$ random matrix having the same distribution as $A$ such that
$2$-dimensional vectors $(a_{ij},a_{ij}')$ ($1\le i\le N$, $1\le j\le n$)
are i.i.d.\ and for any admissible $i$ and $j$ the variables
$a_{ij}$ and $a_{ij}'$ are conditionally i.i.d.\ given event $\{\omega\in\Omega:\,a_{ij}(\omega)\in H\}$ and identical on
$\{\omega\in\Omega:\,a_{ij}(\omega)\in\overline H\}$.
For every $i=1,2,\dots,N$ and $j=1,2,\dots,n$, by the formula for the joint distribution of $a_{ij}$ and $a_{ij}'$ we get
$$
\P\bigl\{\bigl|\trunk{(a_{ij})}{H}-\trunk{(a_{ij}')}{H}\bigr|\ge d\bigr\}
\ge\P\{a_{ij} \in H_1\mbox{ and }a_{ij}'\in H_2\}+\P\{a_{ij} \in H_2\mbox{ and }a_{ij}'\in H_1\}\ge r,
$$
hence, in view of symmetric distribution of $\trunk{(a_{ij})}{H}-\trunk{(a_{ij}')}{H}$, we have
$\concf\bigl(\trunk{(a_{ij})}{H}-\trunk{(a_{ij}')}{H},\frac{d}{2}\bigr)\le 1-\frac{r}{2}$.
Clearly, $h_{\ref{distance estimate}}\ge \frac{d|y_j|}{2}$ for every coordinate $y_j$ of the vector $y$,
hence by Theorem~\ref{rogozin lemma} for all $i=1,2,\dots,N$
\begin{align*}
\P&\Bigl\{\Bigl|\sum\limits_{j=1}^n \bigl(\trunk{(a_{ij})}{H}-\trunk{(a_{ij}')}{H}\bigr) y_j\Bigr|\le h_{\ref{distance estimate}}\Bigr\}\\
&\le \concf\Bigl(\sum\limits_{j=1}^n \bigl(\trunk{(a_{ij})}{H}-\trunk{(a_{ij}')}{H}\bigr) y_j, h_{\ref{distance estimate}}\Bigr)\\
&\le C_{\ref{rogozin lemma}}h_{\ref{distance estimate}}\Bigl(\frac{1}{4}\sum\limits_{j=1}^n
\Bigl(1-\concf\Bigl(\bigl(\trunk{(a_{ij})}{H}-\trunk{(a_{ij}')}{H}\bigr) y_j,\frac{|y_j|d}{2}\Bigr)\Bigr)(y_j d)^2\Bigr)^{-1/2}\\
&\le C_{\ref{rogozin lemma}}h_{\ref{distance estimate}}\Bigl(\frac{r}{8}\sum\limits_{j=1}^n(y_j d)^2\Bigr)^{-1/2}\\
&\le \frac{C_{\ref{rogozin lemma}}h_{\ref{distance estimate}}}{td}\sqrt{\frac{8}{r}}= 1-\delta^{-1/4}.
\end{align*}
Thus, vector $y$ satisfies condition \eqref{new app aux 125} with $s:=h_{\ref{distance estimate}}$.
Then, by Lemma~\ref{wrap signum lemma},
$$\P\bigl\{\dist\bigl(\trunk{A}{H}y,\thesubspace{A}{B}{H}{E}\bigr)
\le h_{\ref{wrap signum lemma}}h_{\ref{distance estimate}}\sqrt{N}\bigr\}\le 2\exp(-w_{\ref{wrap signum lemma}} N).$$
\end{proof}

\section{Decomposition of $S^{n-1}$ and proof of Theorem~\ref{ssv theor nonsym}}

Recall that in Section~\ref{prelim sec} we defined $S^{n-1}_p(\theta)$ as the set of {\it $\theta$-peaky} vectors,
that is, unit vectors in $\R^n$ whose $\ell_\infty^n$-norm is at least $\theta$. 
We say that a vector $y\in S^{n-1}$ is {\it $m$-sparse} if $|\supp y|\le m$.
Next, $y\in S^{n-1}$ is {\it almost $m$-sparse}, if there is a subset $J\subset\{1,2,\dots,n\}$ of
cardinality at most $m$, such that $\|y\chi_{J}\|\ge 1/2$. The set of all almost $m$-sparse
vectors shall be denoted by $S^{n-1}_a(m)$.

In our proof of Theorem~\ref{ssv theor nonsym},
we represent $S^{n-1}$ as the union of three subsets:
$$S^{n-1}=S_p^{n-1}(\theta)\cup\bigl(S^{n-1}_a(\sqrt{N})\setminus S_p^{n-1}(\theta)\bigr)
\cup \bigl(S^{n-1}\setminus S^{n-1}_a(\sqrt{N})\bigr),$$
where $\theta$ is a function of the parameters $\beta$ and $\delta$ of the theorem.
Then the smallest singular value of $A+B$ can be estimated by bounding separately $\inf\limits_y \|Ay+By\|$
over each of the three subsets. 

The reasons for such a representation of $S^{n-1}$ are purely technical: Proposition~\ref{distance estimate}
proved in the previous section handles vectors with a sufficiently small $\ell_\infty^n$-norm, so instead we use
Proposition~\ref{peaky lemma} to deal with the set $S_p^{n-1}(\theta)$.
Further, the separate treatment of almost $\sqrt{N}$-sparse vectors is convenient because, on the one hand,
the construction of the set $H$ corresponding to $S^{n-1}_a(\sqrt{N})\setminus S_p^{n-1}(\theta)$
is trivial compared to $S^{n-1}\setminus S^{n-1}_a(\sqrt{N})$; on the other hand, vectors
from $S^{n-1}\setminus S^{n-1}_a(\sqrt{N})$ have a useful geometric property (Lemma~\ref{spread explained})
which the almost sparse vectors generally do not possess. We note that the set $S^{n-1}_a(\sqrt{N})$ in
the covering of $S^{n-1}$
can be replaced with $S^{n-1}_a(N^\kappa)\bigr)$ for any constant power $\kappa\in(0,1)$; this would
only affect the constants in the final estimate.

In our representation of $S^{n-1}$, we follow an idea from \cite{LPRT}, where
the unit sphere was split into sets of ``close to sparse'' and ``far from sparse'' vectors.
A similar splitting was also employed in \cite{RV_RECT}, \cite{RV_SQUARE},
where the terms ``compressible'' and ``incompressible'' were used instead. On the other hand,
our ``borderline'' $\sqrt{N}$ is smaller by the order of magnitude than in the mentioned papers.

The next elementary lemma shall be used in conjunction with Proposition~\ref{gen spl proc}.
\begin{lemma}\label{net sparsification lemma}
There is a universal constant $C_{\ref{net sparsification lemma}}>0$ with the following property:
Let $n,m\in\N$ with $m\le n$, $\varepsilon\in(0,1]$, $S\subset S^{n-1}$ and
let $T\subset B_2^n$ be a subset of $m$-sparse vectors satisfying
\begin{equation}\label{former sparsif}
\mbox{for any $y\in S$ there is $x=x(y)\in T$ with $y\chi_{\supp x}=x$.}
\end{equation}
Then there is a finite set $\Net\subset T$
of cardinality at most $\bigl(\frac{C_{\ref{net sparsification lemma}}n}{\varepsilon m}\bigr)^m$ such that
for any $y\in S$ there is $y'=y'(y)\in\Net$ with $\|y\chi_{\supp y'}-y'\|\le \varepsilon$.
\end{lemma}
\begin{proof}
For any $J\subset\{1,2,\dots,n\}$ with $|J|\le m$, let $\Net_J$ be an $\varepsilon$-net
for $T\cap\spn\{e_i\}_{i\in J}$ of cardinality at most $\bigl(\frac{3}{\varepsilon}\bigr)^m$. Define $\Net$ as the union
of $\Net_J$ for all admissible $J$. Then, obviously,
$$|\Net|\le 2^m{n\choose m}\Bigl(\frac{3}{\varepsilon}\Bigr)^m
\le\Bigl(\frac{6ne}{\varepsilon m}\Bigr)^m.$$
Next, fix any $y\in S$ and let $x\in T$ be such that $y\chi_{\supp x}=x$. Since $|\supp x|\le m$,
there is $y'\in \Net_{\supp x}\subset\Net$ with $\|x-y'\|\le\varepsilon$.
It remains to note that since $\supp y'\subset \supp x$, necessarily $\|y\chi_{\supp y'}-y'\|\le \|y\chi_{\supp x}-y'\|=\|x-y'\|\le\varepsilon$.
\end{proof}

\begin{prop}[Vectors from $S^{n-1}_a(\sqrt{N})$ with a small $\ell_\infty^n$-norm]\label{compressible lemma}
For any $\gamma>0$ and $\delta>1$ there are
$N_{\ref{compressible lemma}}\in\N$ and $h_{\ref{compressible lemma}}>0$
depending only on $\gamma$ and $\delta$ with the following property:
Let
$$\theta_{\ref{compressible lemma}}
=\frac{1-\delta^{-1/4}}{C_{\ref{rogozin lemma}}}\sqrt{\frac{\gamma}{8}},$$
$N\ge \max(N_{\ref{compressible lemma}}, \delta n)$,
$z\in\R$ and let $A$ be an $N\times n$ random matrix with i.i.d.\ entries such that
$$\min\bigl(\P\bigl\{z-\sqrt{N}\le a_{11}\le z-1\bigr\},
\P\bigl\{z+1\le a_{11}\le z+\sqrt{N}\bigr\}\bigr)\ge\gamma.$$
Then for the set $S=S_a^{n-1}(\sqrt{N})\setminus S_{p}^{n-1}(\theta_{\ref{compressible lemma}})$
and any non-random $N\times n$ matrix $B$
$$\P\bigl\{\inf\limits_{y\in S}\|Ay+By\|\le h_{\ref{compressible lemma}}\sqrt{N}\bigr\}
\le \exp(-w_{\ref{wrap signum lemma}}N/2).$$
\end{prop}
\begin{proof}
Fix any $\gamma>0$ and $\delta>1$ and define
$d:=2$, $r:=\gamma$, $t:=\frac{1}{2}$; let $h_{\ref{distance estimate}}$ be as in Proposition~\ref{distance estimate}
and $N_{\ref{compressible lemma}}=N_{\ref{compressible lemma}}(\gamma,\delta)$
be the smallest integer greater than
$\frac{2}{h_{\ref{wrap signum lemma}}h_{\ref{distance estimate}}}$
such that for all $N\ge N_{\ref{compressible lemma}}$
$$2\bigl(C_{\ref{net sparsification lemma}}N\bigr)^{3\sqrt{N}}
\le \exp(w_{\ref{wrap signum lemma}} N/2).$$
Now, take any $n\in N$ and $N\ge \max(N_{\ref{compressible lemma}},\delta n)$; let $z$ and $A$ safisfy conditions of the lemma
and $B$ be any non-random $N\times n$ matrix. We will assume that $S$ is non-empty.
Without loss of generality, $z=0$ (otherwise, we replace $A$, $B$ with $A-z\mofones$, $B+z\mofones$).
Define $H_1=[-\sqrt{N},-1]$, $H_2=[1,\sqrt{N}]$,
$H=H_1\cup H_2$. Obviously, $\dist(H_1,H_2)= d$ and $\min\bigl(\P\{a_{11}\in H_1\},\P\{a_{11}\in H_2\}\bigr)\ge r$.
Let $T\subset B_2^n$ be the set of $\sqrt{N}$-sparse vectors with the Euclidean norm at least $\frac{1}{2}$
and the maximal norm at most $\theta_{\ref{compressible lemma}}$.
Clearly, $T$ and $S$ satisfy \eqref{former sparsif}, hence, by Lemma~\ref{net sparsification lemma},
there is a finite subset $\Net\subset T$ of cardinality at most
$\bigl(C_{\ref{net sparsification lemma}}N\bigr)^{3\sqrt{N}}$ such that
for any $y\in S$ there is $y'=y'(y)\in\Net$ with $\|y\chi_{\supp y'}-y'\|\le N^{-2}$.

Let $E_{y'}=\spn\{e_j\}_{j\in\supp y'}$ ($y'\in\Net$) and define an event
$$
\mathcal E=\bigl\{\omega\in\Omega:\,\dist\bigl(\trunk{A}{H}(\omega)y', \thesubspace{A}{B}{H}{E_{y'}}(\omega)\bigr)
> h_{\ref{wrap signum lemma}}h_{\ref{distance estimate}}\sqrt{N}\mbox{ for all $y'\in\Net$}\bigr\}.
$$
In view of Proposition~\ref{distance estimate}, the upper estimate for $|\Net|$ and the definition of $N_{\ref{compressible lemma}}$
$$\P(\mathcal E)\ge 1-2|\Net|\exp(-w_{\ref{wrap signum lemma}} N)\ge 1-\exp(-w_{\ref{wrap signum lemma}} N/2).$$
Take any $\omega\in\mathcal E$ and define $\regm=\trunk{A}{H}(\omega)$,
$\irregm=\trunk{A}{\overline H}(\omega)+B$, $D=\regm+\irregm$.
Since all entries of $\regm$ are bounded by $\sqrt{N}$ by absolute value,
we get $\|\regm\|\le N^{3/2}$; next, for every $y'\in\Net$
$$\dist\bigl(\regm y',D(E_{y'}^\perp)+\irregm(E_{y'})\bigr)
> h_{\ref{wrap signum lemma}}h_{\ref{distance estimate}}\sqrt{N}$$
(note that $D(E_{y'}^\perp)+\irregm(E_{y'})=\thesubspace{A}{B}{H}{E_{y'}}(\omega)$).
Hence, by Proposition~\ref{gen spl proc}, we get
$$\inf\limits_{y\in S}\|Dy\|>h_{\ref{wrap signum lemma}}h_{\ref{distance estimate}}\sqrt{N}
-N^{-1/2}\ge\frac{1}{2}h_{\ref{wrap signum lemma}}h_{\ref{distance estimate}}\sqrt{N}.$$
Finally, applying the above argument to all $\omega\in\mathcal E$, we get the result.
\end{proof}

As we noted before, construction of the set $H$ corresponding to $S^{n-1}\setminus S^{n-1}_a(\sqrt{N})$
is not so trivial as in the case of almost $\sqrt{N}$-sparse vectors. The reason is 
that in general the set $S^{n-1}\setminus S^{n-1}_a(\sqrt{N})$ is much larger than
$S^{n-1}_a(\sqrt{N})$, and we have to apply more delicate arguments to get a satisfactory probabilistic estimate.
The construction of $H$ for the set of ``far from $\sqrt{N}$-sparse'' vectors is contained in the following lemma:

\begin{lemma}\label{interval detection lemma}
Let $\xi$ be a random variable such that for some $z\in\R$,
$\gamma>0$, $N\in\N$ we have
$$\min\bigl(\P\bigl\{z-\sqrt{N}\le \xi\le z-1\bigr\},
\P\bigl\{z+1\le \xi\le z+\sqrt{N}\bigr\}\bigr)\ge\gamma.$$
Then there exists an integer $\ell\in[0,\lfloor\log_2 \sqrt{N}\rfloor]$, $\lambda\in\R$ and
disjoint Borel sets $H_1,H_2\subset[-2^{\ell+2};2^{\ell+2}]$ such that
$\dist(H_1,H_2)\ge 2^\ell$,
$\min\bigl(\P\{\xi-\lambda\in H_1\},\P\{\xi-\lambda\in H_2\}\bigr)\ge c_{\ref{interval detection lemma}}\gamma 2^{-\ell/8}$
and
$\Exp\trunk{(\xi-\lambda)}{H}=0$ for $H=H_1\cup H_2$ and a universal constant $c_{\ref{interval detection lemma}}>0$.
\end{lemma}
\begin{proof}
Without loss of generality we can assume that $z=0$.
Let $c_{\ref{interval detection lemma}}=\Bigl(\sum\limits_{m=0}^\infty 2^{-m/8}\Bigr)^{-1}$.
Then, by the conditions on $\xi$, there are
$\ell_1,\ell_2\in\{0,1,\dots,\lfloor\log_2 \sqrt{N}\rfloor\}$ such that
$$\P\{\xi\in[-2^{\ell_1+1},-2^{\ell_1}]\}\ge c_{\ref{interval detection lemma}}\gamma 2^{-\ell_1/8};
\;\;\P\{\xi\in[2^{\ell_2},2^{\ell_2+1}]\}\ge c_{\ref{interval detection lemma}}\gamma 2^{-\ell_2/8}.$$
Now, define $\lambda$ as the conditional expectation of $\xi$ given the event
$\mathcal M=\{\omega\in\Omega:\xi(\omega)\in[-2^{\ell_1+1},-2^{\ell_1}]\cup[2^{\ell_2},2^{\ell_2+1}]\}$, i.e.\
$$\lambda=\P(\mathcal M)^{-1}\int\limits_{\mathcal M}\xi(\omega)d\,\omega.$$
Let $H_1=-\lambda+[-2^{\ell_1+1},-2^{\ell_1}]$ and $H_2=-\lambda+[2^{\ell_2},2^{\ell_2+1}]$. Note that necessarily $\lambda\in[-2^{\ell_1+1},2^{\ell_2+1}]$, hence
$H_1,H_2\subset [-2^{\ell+2},2^{\ell+2}]$ for $\ell=\max(\ell_1,\ell_2)$. Obviously, $\dist(H_1,H_2)\ge 2^\ell$ and for $H=H_1\cup H_2$
$$\Exp\trunk{(\xi-\lambda)}{H}=\int\limits_{\{\xi-\lambda\in H\}}(\xi(\omega)-\lambda)d\,\omega
=\int\limits_{\mathcal M}(\xi(\omega)-\lambda)d\,\omega=0.$$
Finally,
\begin{align*}
\min\bigl(\P\{\xi-\lambda\in H_1\},\P\{\xi-\lambda\in H_2\}\bigr)
&=\min\bigl(\P\{\xi\in [-2^{\ell_1+1},-2^{\ell_1}]\},\P\{\xi\in [2^{\ell_2},2^{\ell_2+1}]\}\bigr)\\
&\ge c_{\ref{interval detection lemma}}\gamma 2^{-\ell/8}.
\end{align*}
\end{proof}

Let us recall a folklore estimate of the norm of a random matrix with bounded mean zero entries
(see, for example, \cite[Proposition~2.4]{RV_CONGRESS}):
\begin{lemma}\label{laplace transform lemma}
Let $W=(w_{ij})$ be an $N\times n$ ($N\ge n$) random matrix with i.i.d.\ mean zero entries; $R>0$ and assume that
$|w_{ij}|\le R$ a.s. Then for a universal constant $C_{\ref{laplace transform lemma}}>0$
$$\P\bigl\{\|W\|\ge C_{\ref{laplace transform lemma}}R\sqrt{N}\bigr\}\le\exp(-N).$$
\end{lemma}

The next lemma highlights a useful property of the vectors from $S^{n-1}\setminus S_a^{n-1}(\sqrt{N})$:

\begin{lemma}\label{spread explained}
For any integer $N\ge n\ge m\ge 1$ and any $y\in S^{n-1}\setminus S_a^{n-1}(\sqrt{N})$ there
is a set $J=J(y)\subset\{1,2,\dots,n\}$ such that $|J|\le m$, $\|y\chi_J\|\ge\frac{1}{2}\sqrt{\frac{m}{n}}$
and $\|y\chi_J\|_\infty\le \frac{1}{\lfloor N^{1/4}\rfloor}$.
\end{lemma}
\begin{proof}
Take any $N\ge n\ge m\ge 1$ and $y=(y_1,y_2,\dots,y_n)\in S^{n-1}\setminus S_a^{n-1}(\sqrt{N})$ and let
$$J'(y)=\Bigl\{j\in\{1,2,\dots,n\}:\,|y_j|\le \frac{1}{\lfloor N^{1/4}\rfloor}\Bigr\}.$$
Obviously, $|J'|\ge n-\sqrt{N}>0$ and, since $y$ is not almost $\sqrt{N}$-sparse,
$\|y\chi_{J'}\|\ge \sqrt{3/4}$. Let $\{J'_1,J'_2,\dots,J'_p\}$ be any partition of $J'$ into pairwise disjoint
subsets of cardinality at most $m$ with $p\le \lceil n/m\rceil$. Then, clearly, for some $q\in\{1,2,\dots,p\}$,
$\|y\chi_{J_q}\|\ge\|y\chi_{J'}\|/\sqrt{p}>\frac{1}{2}\sqrt{\frac{m}{n}}$.
Setting, $J(y)=J_q$, we get the result.
\end{proof}

\begin{prop}[The set $S^{n-1}\setminus S_a^{n-1}(\sqrt{N})$]\label{incompressible vectors lemma}
For any $\gamma>0,\delta>1$ there are $N_{\ref{incompressible vectors lemma}}\in\N$ and
$h_{\ref{incompressible vectors lemma}}>0$ depending only on $\gamma$ and $\delta$
with the following property:
Let $N\ge \max(N_{\ref{incompressible vectors lemma}}, \delta n)$ and let $A$ be an $N\times n$ random matrix
with i.i.d.\ entries such that
$$\min\bigl(\P\bigl\{z-\sqrt{N}\le a_{11}\le z-1\bigr\},
\P\bigl\{z+1\le a_{11}\le z+\sqrt{N}\bigr\}\bigr)\ge\gamma$$
for some $z\in\R$.
Then for any non-random $N\times n$ matrix $B$ and the set $S=S^{n-1}\setminus S_a^{n-1}(\sqrt{N})$ we have
$$\P\bigl\{\inf\limits_{y\in S}\|Ay+By\|\le h_{\ref{incompressible vectors lemma}}\sqrt{N}\bigr\}
\le \exp(-w_{\ref{wrap signum lemma}}N/2).$$
\end{prop}
\begin{proof}
Fix any $\gamma>0$ and $\delta>1$. To make the notation more compact, denote
$f_0:=\frac{(1-\delta^{-1/4})\sqrt{c_{\ref{interval detection lemma}}\gamma}}{C_{\ref{rogozin lemma}}}$ and
let $\tau_0=\tau_0(\gamma,\delta)$ be the largest number in $(0,1]$ such that for all $s\ge 0$
$$\Bigl(\frac{16\sqrt{8}C_{\ref{net sparsification lemma}}C_{\ref{laplace transform lemma}}
2^{s/2}}{h_{\ref{wrap signum lemma}}f_0\tau_0^{3/2}}\Bigr)^{2^{-s/4}\tau_0}
\le\exp(w_{\ref{wrap signum lemma}}/4)$$
(it is not difficult to see that $\tau_0$ is well defined).
Then, take $N_{\ref{incompressible vectors lemma}}=N_{\ref{incompressible vectors lemma}}(\gamma,\delta)$
to be the smallest positive integer such that for all $N\ge N_{\ref{incompressible vectors lemma}}$
\begin{equation}\label{aux incomp 6238}
\frac{1}{\lfloor N^{1/4}\rfloor}\le
\frac{f_0\sqrt{\tau_0}}{4\sqrt{8}}N^{-3/16}
\;\;\mbox{and}\;\;
\frac{48\sqrt{8N}C_{\ref{net sparsification lemma}}C_{\ref{laplace transform lemma}}}
{h_{\ref{wrap signum lemma}}f_0\tau_0^{3/2}}\le \exp(w_{\ref{wrap signum lemma}}N/4).
\end{equation}
Let $N\ge N_{\ref{incompressible vectors lemma}}$, $N\ge\delta n$ and
let $A$ be an $N\times n$ random matrix with entries satisfying conditions of the lemma and $B$
be any non-random $N\times n$ matrix.

By Lemma~\ref{interval detection lemma}, there is an integer $\ell\in[0,\lfloor \log_2\sqrt{N}\rfloor]$,
$\lambda\in\R$ and disjoint Borel sets $H_1,H_2\subset [-2^{\ell+2},2^{\ell+2}]$ such that
$\dist(H_1,H_2)\ge 2^\ell$,
$\min\bigl(\P\{a_{11}-\lambda\in H_1\},\P\{a_{11}-\lambda\in H_2\}\bigr)\ge
c_{\ref{interval detection lemma}}\gamma 2^{-\ell/8}$
and
$\Exp\trunk{(a_{11}-\lambda)}{H}=0$ for $H=H_1\cup H_2$.
Denote $\tilde A=A-\lambda\mofones$, $\tilde B=B+\lambda\mofones$ and let
$$R:=2^{\ell+2},\;d:=2^\ell,\;r:=c_{\ref{interval detection lemma}}\gamma 2^{-\ell/8},\;
m:=\Bigl\lceil\frac{\tau_0 n}{2^{\ell/4}}\Bigr\rceil,\;t:=\frac{1}{2}\sqrt{\frac{m}{n}},\;
\varepsilon:=\frac{h_{\ref{wrap signum lemma}}h_{\ref{distance estimate}}}{2C_{\ref{laplace transform lemma}}R},$$
where $h_{\ref{distance estimate}}$ is defined as in Proposition~\ref{distance estimate}.
Assume that $S$ is non-empty and let $T\subset B_2^n$ consist of all $m$-sparse vectors $y\in B_2^n$ with $\|y\|\ge t$
and $\|y\|_\infty\le \frac{2h_{\ref{distance estimate}}}{d}$.
The first inequality in \eqref{aux incomp 6238} and a simple calculation show that
$\frac{1}{\lfloor N^{1/4}\rfloor}\le \frac{2h_{\ref{distance estimate}}}{d}$.
Hence, in view of Lemma~\ref{spread explained},
$T$ is non-empty and satisfies \eqref{former sparsif}. By Lemma~\ref{net sparsification lemma},
there is a finite subset $\Net\subset T$ of cardinality at most $\bigl(\frac{nC_{\ref{net sparsification lemma}}}{m\varepsilon}\bigr)^{m}$
such that for any $y\in S$ there is $y'=y'(y)\in\Net$ with $\|y\chi_{\supp y'}-y'\|\le \varepsilon$.

For each $y'\in\Net$ denote $E_{y'}=\spn\{e_j\}_{j\in\supp y'}$. By Proposition~\ref{distance estimate},
$$\P\bigl\{\dist\bigl(\trunk{\tilde A}{H}y',\thesubspace{\tilde A}{\tilde B}{H}{E_{y'}}\bigr)
\le h_{\ref{wrap signum lemma}}h_{\ref{distance estimate}}\sqrt{N}\bigr\}\le 2\exp(-w_{\ref{wrap signum lemma}} N).$$
Define an event
\begin{align*}
\mathcal E
=\bigl\{&\omega\in\Omega:\,
\dist\bigl(\trunk{\tilde A}{H}(\omega)y, \thesubspace{\tilde A}{\tilde B}{H}{E_{y'}}(\omega)\bigr)
> h_{\ref{wrap signum lemma}}h_{\ref{distance estimate}}\sqrt{N}\\
&\mbox{for all $y'\in\Net$ and $\|\trunk{\tilde A}{H}(\omega)\|\le C_{\ref{laplace transform lemma}}R\sqrt{N}$}\bigr\}.
\end{align*}
By the above probability estimates and Lemma~\ref{laplace transform lemma},
$$
\P(\mathcal E)\ge 1-\exp(-N)-2|\Net|\exp\bigl(-w_{\ref{wrap signum lemma}}N\bigr)
\ge 1-\exp(-N)-2\Bigl(\frac{C_{\ref{net sparsification lemma}}n}{m\varepsilon}\Bigr)^{m}\exp\bigl(-w_{\ref{wrap signum lemma}}N\bigr).
$$
Using the definition of $\varepsilon$, $m$, $\tau_0$ and the second inequality in \eqref{aux incomp 6238}, we can estimate the probability as
\begin{align*}
\P(\mathcal E)
&\ge 1-3\Bigl(\frac{8C_{\ref{net sparsification lemma}}C_{\ref{laplace transform lemma}}
2^{\ell+\ell/4}}{\tau_0 h_{\ref{wrap signum lemma}}h_{\ref{distance estimate}}}\Bigr)^{2^{-\ell/4}\tau_0 n+1}
\exp(-w_{\ref{wrap signum lemma}} N)\\
&\ge 1-3\Bigl(\frac{16\sqrt{8}C_{\ref{net sparsification lemma}}C_{\ref{laplace transform lemma}}
2^{\ell/2}}{h_{\ref{wrap signum lemma}}f_0\tau_0^{3/2}}\Bigr)^{2^{-\ell/4}\tau_0 n+1}
\exp(-w_{\ref{wrap signum lemma}} N)\\
&\ge 1-\exp(-w_{\ref{wrap signum lemma}} N/2).
\end{align*}

Take any $\omega\in\mathcal E$ and define $\regm=\trunk{\tilde A}{H}(\omega)$,
$\irregm=\trunk{\tilde A}{\overline H}(\omega)+\tilde B$, $D=A(\omega)+B(\omega)=\regm+\irregm$.
Then $\|\regm\|\le C_{\ref{laplace transform lemma}}R\sqrt{N}$ and for every $y'\in\Net$ we have
$$\dist\bigl(\regm y',D(E_{y'}^\perp)+\irregm(E_{y'})\bigr)
> h_{\ref{wrap signum lemma}}h_{\ref{distance estimate}}\sqrt{N}.$$
Hence, by Proposition~\ref{gen spl proc} and the definition of $\varepsilon$, we get
$$\inf\limits_{y\in S}\|Dy\|>h_{\ref{wrap signum lemma}}h_{\ref{distance estimate}}\sqrt{N}
-\varepsilon C_{\ref{laplace transform lemma}}R\sqrt{N}
= \frac{1}{2}h_{\ref{wrap signum lemma}}h_{\ref{distance estimate}}\sqrt{N}
\ge\frac{h_{\ref{wrap signum lemma}}f_0\sqrt{\tau_0}}{4\sqrt{8}}\,\sqrt{N}.$$
Finally, applying the above argument to entire set $\mathcal E$, we obtain the result.
\end{proof}

\begin{proof}[Proof of Theorem~\ref{ssv theor nonsym}]
In view of the trivial identity $\concf(a_{ij},\alpha)=\concf(a_{ij}/\alpha,1)$, 
it is enough to prove the theorem for $\alpha=1$. Fix any $\delta>0$ and $\beta>0$, let $\gamma=\beta/4$ and let
$N_0=N_0(\beta,\delta)$ be the smallest integer such that
$N_0\ge \max(N_{\ref{compressible lemma}}, N_{\ref{incompressible vectors lemma}})$ and for all $N\ge N_0$
$$N\le\exp(w_{\ref{peaky lemma}}N/2)\;\;\mbox{and}\;\;3\le \exp\bigl(\min(w_{\ref{peaky lemma}},w_{\ref{wrap signum lemma}})N/4\bigr).$$
Take any $N,n\in\N$ with $N\ge \max(N_0,\delta n)$, let $A=(a_{ij})$ be a $N\times n$ random matrix with i.i.d.\
entries satisfying $\concf(a_{11},1)\le 1-\beta$ and let $B$ be any non-random $N\times n$ matrix.
By the right-continuity of the cdf of $a_{11}$, there is $z\in\R$ such that
$$\P\{a_{11}\le z-1\}\ge\frac{\beta}{2}\;\mbox{ and }\;\P\{a_{11}<z-1\}\le\frac{\beta}{2}.$$
Then
$$\P\{a_{11}\ge z+1\}\ge 1-\P\{a_{11}<z-1\}-\concf(a_{11},1)\ge\frac{\beta}{2}.$$
Let us consider three cases.

{\bf 1)} $\P\{z+1\le a_{11}\le z+\sqrt{N}\}\le \gamma$. Then
$\concf(a_{11},\sqrt{N}/8)\le \concf(a_{11},(\sqrt{N}-1)/2)\le 1-\gamma$.
Obviously, any vector on $S^{n-1}$ is $N^{-1/2}$-peaky. Then, applying Proposition~\ref{peaky lemma}
with the ``scaling factor'' $\sqrt{N}/8$, we get
\begin{align*}
\P\bigl\{s_n(A+B)\le h_{\ref{peaky lemma}}\sqrt{N}/8\bigr\}
&=\P\bigl\{\inf\limits_{y\in S^{n-1}}\|Ay+By\|\le h_{\ref{peaky lemma}}\sqrt{N}/8\bigr\}\\
&\le n\exp(-w_{\ref{peaky lemma}}N)\\
&\le\exp(-w_{\ref{peaky lemma}}N/2).
\end{align*}

{\bf 2)} $\P\{z-\sqrt{N}\le a_{11}\le z-1\}\le\gamma$. Treated as above.

{\bf 3)} $\min\bigl(\P\{z-\sqrt{N}\le a_{11}\le z-1\},\P\{z+1\le a_{11}\le z+\sqrt{N}\}\bigr)\ge \gamma$.
Define $\theta_{\ref{compressible lemma}}$ as in Proposition~\ref{compressible lemma}. By Proposition~\ref{peaky lemma} for peaky vectors,
$$\P\bigl\{\inf\limits_{y\in S_p^{n-1}(\theta_{\ref{compressible lemma}})}\|Ay+By\|
\le h_{\ref{peaky lemma}}\theta_{\ref{compressible lemma}} \sqrt{N}\bigr\}\le n\exp(-w_{\ref{peaky lemma}}N)
\le\exp(-w_{\ref{peaky lemma}}N/2).$$
By Propositions~\ref{compressible lemma} and~\ref{incompressible vectors lemma} for
$S=S_a^{n-1}(\sqrt{N})\setminus S_p^{n-1}(\theta_{\ref{compressible lemma}})$ and
$S'=S^{n-1}\setminus S_a^{n-1}(\sqrt{N})$ we have
\begin{align*}
&\P\bigl\{\inf\limits_{y\in S}\|Ay+By\|\le h_{\ref{compressible lemma}}\sqrt{N}\bigr\}\le \exp(-w_{\ref{wrap signum lemma}}N/2);\\
&\P\bigl\{\inf\limits_{y\in S'}\|Ay+By\|\le h_{\ref{incompressible vectors lemma}}\sqrt{N}\bigr\}\le \exp(-w_{\ref{wrap signum lemma}}N/2).
\end{align*}
Combining the estimates, we get for
$h=\min\bigl(h_{\ref{peaky lemma}}\theta_{\ref{compressible lemma}},
h_{\ref{compressible lemma}},h_{\ref{incompressible vectors lemma}}\bigr)$:
\begin{align*}
\P\bigl\{s_n(A+B)\le h\sqrt{N}\bigr\}&\le \exp(-w_{\ref{peaky lemma}}N/2)+2\exp(-w_{\ref{wrap signum lemma}}N/2)\\
&\le\exp\bigl(-\min(w_{\ref{peaky lemma}},w_{\ref{wrap signum lemma}})N/4\bigr).
\end{align*}
This completes the proof.
\end{proof}

\section{Acknowledgement}

I would like to thank my supervisor Prof. N.~Tomczak-Jaegermann for valuable suggestions that helped improve structure of the proof.

\end{document}